\documentclass[12pt]{article}
\usepackage{amssymb,amsfonts,amsmath,amsthm}
\usepackage{enumerate,color,xcolor}
\usepackage{setspace}
\usepackage{natbib}
\newcommand{\ee}{{\mathbb E}}

\newcommand{\nn}{{\mathbb N}}
\newcommand{\pp}{{\mathbb P}}
\newcommand{\rr}{{\mathbb R}}

\newcommand{\beq}{\begin{eqnarray*}}
\newcommand{\feq}{\end{eqnarray*}}
\newcommand{\beqn}{\begin{eqnarray}}
\newcommand{\feqn}{\end{eqnarray}}

\newtheorem{theorem}{Theorem}
\makeatletter \@addtoreset{theorem}{section}\makeatother

\newtheorem{lemma}[theorem]{Lemma}
\newtheorem{definition}[theorem]{Definition}

\newtheorem*{theorem*}{Theorem}
\newtheorem{proposition}[theorem]{Proposition}

\newtheorem{remark}[theorem]{Remark}

\begin{document}

\title{Precise deviations for Cox processes with a shot noise intensity}
\author{Zailei Cheng\thanks{Department of Mathematics, Florida State University, Tallahassee, FL 32306, United State of America;e-mail: zcheng@math.fsu.edu}
\and Youngsoo Seol\thanks{Department of Mathematics, Dong-A
University, Busan, Saha-gu, Nakdong-daero 550, 37, Republic of
Korea;e-mail: prosul76@dau.ac.kr}}
\maketitle

\begin{abstract}
We consider a Cox process with Poisson shot noise intensity which
has been widely applied in insurance, finance, queue theory,
statistic, and many other fields. Cox process is flexible because
its intensity not only depends on the time but also can be
considered as a stochastic process and so it can be considered as a
two step randomization procedure. Due to the structure of such
models, a number of useful and general results can easily be
established. In this paper, we study the fluctuations and precise
deviations for shot noise Cox process using the recent mod-$\phi$
convergence method.
\end{abstract}
{\em MSC2010: } primary 60G55; secondary 60F05, 60F10.\\
\noindent{\it Keywords}: Cox processes, Poisson shot noise, Mod$-\phi$ convergence, Precise deviations.
\maketitle

\section{Introduction}
A Cox process is first introduced by \cite{Cox} and a natural
generalization of a Poisson process by considering the intensity of
Poisson process as a realization of a random measure~\citep{Moller}.
The Cox process provides the flexibility of letting the intensity
not only depend on time but also allowing it to be a stochastic
process. Hence, it can be viewed as a two-step randomization
procedure which can deal with the stochastic nature of catastrophic
loss occurrences in the real world. Moreover, shot noise processes
~\citep{Cox2} are particularly useful to model claim arrivals; they
provide measures for frequency, magnitude and the time period needed
to determine the effect of catastrophic events within the same
framework; as time passes, the shot noise process decreases as more
and more losses are settled, and this decrease continues until
another event occurs which will result in a positive jump.
Therefore, the shot noise process can be used as the intensity of a
Cox process to measure the number of catastrophic losses. The shot
noise Cox processes were introduced in~\cite{Moller} and further
generalized in~\cite{Hellmund} without discussing the statistical
inference for the model. Note that the class of shot noise Cox
processes also includes the very popular Poisson Neyman-Scott
processes like the Thomas process(\citep{Thomas},[\citep{Illian},
Section 6.3.2]).

The Cox model has been used widely in many aspects: insurance,
finance, queue theory, statistic
etc.(cf.\citep{Albrecher},\citep{Cox},\citep{Dassios2}). \cite{Dassios2} applied the Cox process with Poisson shot noise
intensity to pricing stop-loss catastrophe reinsurance contract and
catastrophe insurance derivatives. \cite{Albrecher} studied the asymptotic estimates for
infinite time and finite time ruin probabilities of the risk model.

There are many applications of large deviations such as insurance,
portfolio management, risk management, queue system,
statistics(\citep{Macci2}, \citep{Gao},\citep{Shen},\citep{Ganesh}).
\cite{Macci2} studied the large deviation principles for the Markov
modulated risk process with reinsurance. \cite{Gao} extended the
result considered in~\cite{Macci2} to the sample path large and
moderate deviation principles. \cite{Shen} obtained the precise
large deviation results for the customer arrival based risk model
which can be treated as a generalized Poisson shot noise process.
Recently, \cite{Macci} took into account the large deviation
estimations for the ruin probability of the risk processes with shot
noise Cox claim number process and reserve dependent premium rate.
Many researchers have made great efforts of the precise large
deviations for the loss process of a classical insurance risk model
and have obtained a lot of inspiring results. Precise large
deviations for the loss process have been widely investigated. For
the classic results, we refer the reader to \cite{cline},
\cite{kluppelberg}, \cite{mikosch}, and \cite{Ng}, among others. It
is worth mentioning that the results of precise large deviations for
random sums are particulary useful for evaluation of some risk
measures such as conditional tail expectation and value at risk of
aggregate claims of a large insurance portfolio; see \cite{mcneil}
for a review of risk measures. \cite{wang} investigated precise
large deviations of multi-risk models.

Previous works on insurance applications using a shot noise process
or a Cox process with shot noise intensity can be found in
\cite{bremaud}, \cite{Dassios2}, \cite{Jang}, \cite{Torrisi},
\cite{Albrecher}, \cite{Macci}, \cite{Zhu5} and \cite{Schmidt}.
However, previous study on the large deviations and moderate
deviations for Cox process only gives the leading order term
\citep{Gao2}, but not the higher order expansion which make more
accurate computational tractability for an financial application.
Distributional properties for such a model generate computational
tractability for an financial application. Transform, distribution,
and moment formulae computational tractability for a range of
applications in portfolio credit risk. The transform formulae
facilitate the valuation, hedging, and calibration of a portfolio
credit derivative, which is a security whose payoff is a specified
function of the portfolio loss, and which provides insurance against
default losses in the portfolio.

 In this paper, we assume that $\{N_{t},t\geq0\}$ is a
Cox process with Poisson shot noise intensity, that is, the
intensity $\lambda_{t}$ of $N_{t}$ is stochastic and at time t,
given by \beqn\label{intencox}
\lambda_{t}=\nu+\int_{0}^{t}g(t-s)d\bar{N}_{s}, \feqn where
$\bar{N}_{t}$ is a Poisson process with intensity $\rho,$ $\nu>0$
and $g:\rr_{+}\rightarrow\rr_{+}$ being locally bounded. We study
precise deviations for shot noise Cox process using the recent
mod-$\phi$ convergence method developed in~\cite{Feray}. In many
applications in finance, insurance, and other fields, more precise
deviations are desired, which motivates us to study the precise
deviations for Cox process with a shot noise due to higher order
expansion given by the precise deviation for Cox process. Taylor
expansions have numerous applications to simplify complex functions
in finance such as the price of a portfolio of options or bonds. The
more terms used in the expansion the more accurate the
approximation. Higher order Taylor expansions are obtained using
higher partial derivatives. Higher derivatives are derivatives of
derivatives. Thus higher order expansion is one of the most
important concepts in mathematical finance literature. In
particular, using the re-normalization theory called mod-$\phi$
convergence method, we construct an extremely flexible framework to
study the precise deviations and limit theorems. This type of
convergence is a relatively new concept with many deep
ramifications.  It is worth mentioning that the results of precise
large deviations for Cox process with a shot noise are particulary
useful for calculations and evaluations in financial applications
using the higher order expansion given by the precise deviations
which is more applicable than the large or moderate deviations
principle. Very recently, precise deviations for Hawkes processes
was studied in \cite{Gao3} using the mod-$\phi$ convergence method.
In particular, \cite{Gao3} used Hawkes processes for large time
asymptotics, that strictly extends and improves the existing results
in the literature. In order to apply the results of mod-$\phi$
convergence theory, we manipulate the the moment generating function
for shot noise Cox process, and then the precise deviations
principle and precise moderate deviation and fluctuation results are
obtained by the mod-$\phi$ convergence method after careful analysis
and series of lemmas.
\par
The structure of this paper is organized as follows. Some auxiliary
results and the main results are stated in Section 2. The proofs for
the main theorems are contained in Section 3.
\par
\section{Statement of the main results}
This section states the main results of this paper. It consists of
two key lemmas and the precise deviations principle and furthermore
precise deviation principle and fluctuation results. The main
strategy of proving the precise deviations principle is by showing
the mod-$\phi$ convergence as defined in~\cite{Feray} and apply
their Theorem 3.4 and 3.9 to get the main Theorem. We start with the
definition of mod-$\phi$ convergence and then the assumptions which
we will use throughout the paper.

Let us first recall the definition of mod-$\phi$ convergence, see
e.g. Definition 1.1.~\citep{Feray}. Let $(X_{n})_{n\in\nn}$ be a
sequence of real-valued random variables and $\ee[e^{zX_{n}}]$ exist
in a strip
$\mathcal{S}_{(c,d)}:=\{z\in\mathbb{C}:c<\mathcal{R}(z)<d\},$ with
$c<d$ extended real numbers, i.e. we allow $c=-\infty$ and
$d=+\infty$ and $\mathcal{R}(z)$ denotes the real part of $z\in
\mathbb{C}$ throughout this paper. We assume that there exists a
non-constant infinitely divisible distribution $\phi$ with
$\int_{\rr}e^{zx}\phi(dx)=e^{\eta(z)},$ which is well defined on
$\mathcal{S}_{(c,d)},$ and an analytic function $\psi(z)$ that does
not vanish on the real part of $\mathcal{S}_{(c,d)}$ such that
locally uniformly in $z\in\mathcal{S}_{(c,d)},$ \beqn\label{convergence}
e^{-t_{n}\eta(z)}\ee[e^{z X_{n}}]\rightarrow\psi(z), \feqn where
$t_{n}\rightarrow+\infty$ as $n\rightarrow\infty.$ Then we say that
$X_{n}$ converges mod-$\phi$ on $\mathcal{S}_{(c,d)}$ with
parameters $(t_{n})_{n\in\nn}$ and limiting function $\psi.$

Assume that $\phi$ is a lattice distribution and the convergence is at speed $ \mathcal{O}\left(\frac{1}{t^{\nu}_{n}}\right) $. Then Theorem
3.4.~\cite{Feray} says that for any $x\in\rr$ in the interval
$(\eta'(c),\eta'(d))$ and $\theta^{*}$ defined as
$\eta'(\theta^{*})=x,$ assume that $t_{n}x\in\nn,$ then, \beqn
\pp(X_{n}=t_{n}x)=\frac{e^{-t_{n}F(x)}}{\sqrt{2\pi
t_{n}\eta''(\theta^{*})}}\Biggl(\psi(\theta^{*})+\frac{a_{1}}{t_{n}}+\frac{a_{2}}{t^{2}_{n}}+\cdots
+\frac{a_{\nu-1}}{t^{\nu-1}_{n}}+\mathcal{O}\left(\frac{1}{t^{\nu}_{n}}\right)\Biggr),
\feqn as $n\rightarrow\infty,$ where
$F(x):=\sup_{\theta\in\rr}\{\theta x-\eta(\theta)\}$ is the Legendre
transform of $\eta(\cdot),$ and similarly, if $x\in\rr$ is in the
range of $(\eta'(0),\eta'(d)),$ then, \beqn \pp(X_{n}\geq
t_{n}x)=\frac{e^{-t_{n}F(x)}}{\sqrt{2\pi
t_{n}\eta''(\theta^{*})}}\frac{1}{1-e^{-\theta^{*}}}\Biggl(\psi(\theta^{*})+\frac{b_{1}}{t_{n}}+\frac{b_{2}}{t_{n}}+\cdots
+\frac{b_{\nu-1}}{t^{\nu-1}_{n}}+\mathcal{O}\left(\frac{1}{t^{\nu}_{n}}\right)\Biggr),
\feqn as $n\rightarrow\infty,$ where $(a_{k})_{k=1}^{\infty},$
$(b_{k})_{k=1}^{\infty}$ are rational fractions in the derivatives
of $\eta$ and $\psi$ at $\theta^{*},$ that can be computed as
described in Remark 3.7.~\citep{Feray}.

Then let's recall Theorem 3.9.~\citep{Feray}, which states the
precise moderate deviation.

Consider a sequence $ (X_n)_{n\in\mathbb{N}} $ that converges mod-$ \phi $, with a reference infinitely divisible law $ \phi $ that is a lattice distribution. Assume $ y=o((t_n)^{1/6}) $. Then,

\begin{equation}
\mathbb{P}\left(X_n\geq t_n\eta'(0)+\sqrt{t_n\eta''(0)}y\right)=\mathbb{P}\left(\mathcal{N}_\mathbb{R}(0,1)\geq y\right)(1+o(1)).
\end{equation}

On the other hand, assuming $ y\gg 1 $ and $y=o((t_n)^{1/2}) $, if $ x=\eta'(0)+\sqrt{\eta''(0)/t_n}y $ and $ h $ is the solution of $ \eta'(h)=x $, then

\begin{equation}
\mathbb{P}\left(X_n\geq t_n\eta'(0)+\sqrt{t_n\eta''(0)}y\right)=\frac{e^{-t_nF(x)}}{h\sqrt{2\pi t_n \eta''(h)}}(1+o(1))
\end{equation}

And Corollary 3.13~\citep{Feray} gives a more explicit form of Theorem 3.9.~\citep{Feray}. If $y=o((t_n)^{1/4})$, then one has

\begin{equation}
\mathbb{P}\left(X_n\geq t_n\eta'(0)+\sqrt{t_n\eta''(0)}y\right)=\frac{(1+o(1))}{y\sqrt{2\pi}}e^{-\frac{y^2}{2}}\exp\left({\frac{\eta'''(0)}{6(\eta''(0))^{3/2}}\frac{y^3}{\sqrt{t_n}}}\right)
\end{equation}

More generally, if $ y=o((t_n)^{1/2-1/m}) $, then one has

\begin{equation}
\mathbb{P}\left(X_n\geq t_n\eta'(0)+\sqrt{t_n\eta''(0)}y\right)=\frac{(1+o(1))}{y\sqrt{2\pi}}\exp\left(-\sum_{i=2}^{m-1}\frac{F^{(i)}(\eta'(0))}{i!}\frac{(\eta''(0))^{i/2}y^i}{t_n^{(i-2)/2}}\right)
\end{equation}

Here we consider $N_{t}$ as a Cox process with Poisson shot noise
intensity defined in~\eqref{intencox}, we assume throughout this
paper that

\begin{itemize}
    \item $\Vert g\Vert_{L^{1}}=\int_{0}^{\infty}g(t)dt<\infty$
    \item $ g(t)=\mathcal O\left(\frac{1}{t^{\nu+2}}\right) $ as $ t\to\infty $
\end{itemize}

By the definition of mod-$ \phi $ convergence, the moment generating
function of $ N_t $ is finite, which implies the first assumption.
We also need the second assumption for the convergence speed  in
mod-$ \phi $ convergence.

Our main results for the precise large deviations and precise
moderate deviations for the Cox process with Poisson shot noise
intensity are stated as follows.

\subsection{Precise Large Deviations}

\begin{definition}\label{def1}
    We say that the sequence of random variables $ (X_n)_{n\in\mathbb{N}} $ converges mod-$\phi$ at speed $ \mathcal{O}\left(\frac{1}{t^{\nu}_{n}}\right) $ if the difference of the two sides of Equation \ref{convergence} can be bounded by $ C_K\left(\frac{1}{t^{\nu}_{n}}\right) $ for any $ z $ in a given compact subset $ K $ of $\mathcal{S}_{(c,d)}$.
\end{definition}

\begin{theorem}\label{theo1}

(i) For any $x>0,$ and $tx\in\nn,$ as $t\rightarrow\infty,$
\beqn
\pp(N_{t}=tx)=e^{-tI(x)}\sqrt{\frac{I''(x)}{2\pi t}}\Biggl(\psi(\theta^{*})+\frac{a_{1}}{t}+\frac{a_{2}}{t^{2}}+\cdots
+\frac{a_{\nu-1}}{t^{\nu-1}}+\mathcal{O}\left(\frac{1}{t^{\nu}}\right)\Biggr),
\feqn
where for any $\theta\in \mathbb{C}$
\beqn
\psi(\theta):=e^{\rho\varphi(\theta)}
\feqn
and
\beqn
\varphi(\theta)=\int_{0}^{\infty}\Bigl[e^{(e^{\theta}-1)\int_{0}^{u}g(s)ds}-e^{(e^{\theta}-1)\int_{0}^{\infty}g(s)ds}\Bigr]du,
\feqn
which is analytic in $\theta$ for any $\theta\in \mathbb{C}$ and $ I(x) $ is defined as

\begin{equation}\label{Ix}
I(x)=\sup_{\theta\in\rr}\{\theta x-\eta(\theta)\big\}
\end{equation}

and

\begin{equation}\label{theta}
\eta(\theta)=(e^{\theta}-1)\nu+\rho\Bigl(e^{(e^{\theta}-1)\Vert g\Vert_{L^{1}}}-1\Bigr)
\end{equation}
$ \theta^{*} $ is the solution of $ \eta'(\theta^{*})=x $ and
$(a_{k})_{k=1}^{\infty}$ are rational fractions in the derivatives
of $\eta$ and $\psi$ at $\theta^{*}$.

Note that $ \theta^{*} $ is unique for each $ x $ because it is easy to show that $ \eta''(\theta)>0 $.

(ii) For any $x>\nu+\rho\Vert g\Vert_{L_1},$ as $t\rightarrow\infty,$
\beqn
\pp(N_{t}\geq tx)=e^{-tI(x)}\sqrt{\frac{I''(x)}{2\pi t}}\frac{1}{1-e^{-\theta^{*}}}\Biggl(\psi(\theta^{*})+\frac{b_{1}}{t}+\frac{b_{2}}{t^2}+\cdots
+\frac{b_{\nu-1}}{t^{\nu-1}}+\mathcal{O}\left(\frac{1}{t^{\nu}}\right)\Biggr),
\feqn
where $(b_{k})_{k=1}^{\infty}$ are rational fractions in the derivatives of $\eta$ and $\psi$ at $\theta^{*}$.
\end{theorem}
\par
The key to prove main Theorem is to verify the mod-$\phi$ convergence. More precisely, we need to show the following three lemmas.

\begin{lemma}\label{theo2}
    Y has an infinitely divisible distribution. Y is some random variable defined as
    \beqn
    e^{\eta(\theta)}=\exp\Bigl[(e^{\theta}-1)\nu+\rho\Bigl(e^{(e^{\theta}-1)\Vert g\Vert_{L^{1}}}-1\Bigr)\Bigr]=\ee[e^{\theta Y}],
    \feqn
 Note that infinitely divisible is a limitation of the method of mod-$\phi$ convergence. Fortunately, the limiting distribution in the case of the Cox process with shot noise is indeed infinitely divisible.
\end{lemma}
\par
\begin{lemma}\label{theo4}
    For any $ \theta \in \mathbb{C}$,
    \beqn
    \varphi(\theta)=\int_{0}^{\infty}\Bigl[{e^{(e^{\theta}-1)\int_{0}^{u}g(s)ds}}-e^{(e^{\theta}-1)\int_{0}^{\infty}g(s)ds}\Bigr]du,
    \feqn
    is well-defined and analytic in $\theta,$ and
    \beqn
    e^{-t\eta(\theta)}\ee[e^{\theta N_{t}}]\rightarrow\psi(\theta):=e^{\rho\varphi(\theta)},
    \feqn
    as $t\rightarrow\infty,$ locally uniformly in $\theta$.
\end{lemma}

\begin{lemma}\label{theo10}
    \{$ N_t,t\geq0 $\} converges mod-$\phi$ at speed $ \mathcal{O}\left(\frac{1}{t^{\nu}}\right) $ as $ t\to\infty $.
\end{lemma}


\subsection{Precise Moderate Deviations}

By Theorem 3.9.~\citep{Feray}, we have the following theorem:

\begin{theorem}\label{theo5}
    (i)
For any $ y=o(t^{1/6}) $, as $ t\to \infty $,
\begin{equation}
\mathbb{P}\left(N_t\geq(\nu+\rho \Vert g\Vert_{L^{1}})t+\sqrt{t}\sqrt{\nu+\rho\Vert g\Vert_{L^{1}}(\Vert g\Vert_{L^{1}}+1)}y\right)=\Psi(y)(1+o(1)),
\end{equation}
where $ \Psi(y):=\int_{y}^{\infty}\frac{1}{\sqrt{2\pi}}e^{-x^2/2}dx $.

(ii)
For any $ y\gg 1 $ and $ y=o(t^{1/2}) $, as $ t\to \infty $,
\begin{equation}
\mathbb{P}\left(N_t\geq(\nu+\rho \Vert g\Vert_{L^{1}})t+\sqrt{t}\sqrt{\nu+\rho\Vert g\Vert_{L^{1}}(\Vert g\Vert_{L^{1}}+1)}y\right)=\frac{e^{-tI(x^*)}}{\theta^*\sqrt{2\pi t \eta''(\theta^*)}}(1+o(1)),
\end{equation}
where $ x^*=(\nu+\rho \Vert
g\Vert_{L^{1}})+\frac{1}{\sqrt{t}}\sqrt{\nu+\rho\Vert
g\Vert_{L^{1}}(\Vert g\Vert_{L^{1}}+1)}y $ and $ \theta^* $ is the
solution of $ \eta'(\theta^*)=x^* $ and the notation $a\gg b$ means
b is much less than a.

\end{theorem}

\begin{remark}
    Note that Theorem \ref{theo5} (i) is an improvement of the classical central limit theorem because here we allow $ y=o(t^{1/6}) $ for $ t\to \infty $. Theorem \ref{theo5} (ii) is the moderate deviations.
\end{remark}

By using Corollary 3.13~\citep{Feray}, we can get a more explicit form of Theorem \ref{theo5}.

\begin{theorem}\label{theo6}
    (i)
    For any $ y=o(t^{1/4}) $, as $ t\to \infty $,
    \begin{equation}
    \begin{split}
\mathbb{P}\left(N_t\geq(\nu+\rho \Vert
g\Vert_{L^{1}})t+\sqrt{t}\sqrt{\nu+\rho\Vert g\Vert_{L^{1}}(\Vert
g\Vert_{L^{1}}+1)}y\right)\\
=\frac{(1+o(1))}{y\sqrt{2\pi}}e^{-\frac{y^2}{2}}\exp\left({\frac{\eta'''(0)}{6(\eta''(0))^{3/2}}\frac{y^3}{\sqrt{t}}}\right)
    \end{split}
    \end{equation}
    (ii)
    For any $ y=o(t^{1/2-1/m}) $, where $m\geq 3$, as $ t\to \infty $,
    \begin{equation}
    \begin{split}
\mathbb{P}\left(N_t\geq(\nu+\rho \Vert g\Vert_{L^{1}})t+\sqrt{t}\sqrt{\nu+\rho\Vert g\Vert_{L^{1}}(\Vert g\Vert_{L^{1}}+1)}y\right) \\
=\frac{(1+o(1))}{y\sqrt{2\pi}}\exp\left(-\sum_{i=2}^{m-1}\frac{I^{(i)}(\eta'(0))}{i!}\frac{(\eta''(0))^{i/2}y^i}{t^{(i-2)/2}}\right),
    \end{split}
    \end{equation}
    where $I(x)$ is defined in \eqref{Ix} and $ \eta(\theta) $ is defined in \eqref{theta}.
\end{theorem}

\section{Proofs}
In this section, we give a proof of the main theorems and lemmas.
The key idea to prove the main theorem is to apply the recently
developed mod-$\phi$ convergence method.

\subsection{Proofs of the results in Section 2.1}

\begin{proposition}
Assume that $N_{t}$ is Cox process with Poisson shot noise intensity defined in~\eqref{intencox}, then the moment generating function of $N_{t}$ is
\beqn\label{mgf1}
\ee\left[e^{\theta N_{t}}\right]=\exp{\Biggl[(e^{\theta}-1)t\nu+\rho\int_{0}^{t}\Bigl(e^{(e^{\theta}-1)\int_{0}^{u}g(s)ds}-1\Bigr)du\Biggr]}
\feqn
\end{proposition}
\begin{proof}
By the definition of a Cox process with Poisson shot noise intensity and Fubini's theorem, we have
\begin{equation}\label{EY}
\begin{split}
\ee\left[e^{\theta N_{t}}\right]&=\ee\left[e^{(e^{\theta}-1)\int_{0}^{t}\lambda_{s}ds}\right]\\
&=\ee\left[e^{(e^{\theta}-1)\left[\nu t+\int_{0}^{t}\int_{0}^{s}g(s-u)d\bar{N}_{u}ds\right]}\right]\\
&=\ee\left[e^{(e^{\theta}-1)\left[\nu t+\int_{0}^{t}[\int_{u}^{t}g(s-u)ds]d\bar{N}_{u}\right]}\right]\\
&=e^{(e^{\theta}-1)t\nu}\ee\left[e^{(e^{\theta}-1)\int_{0}^{t}[\int_{u}^{t}g(s-u)ds]d\bar{N}_{u}}\right]\\
&=e^{(e^{\theta}-1)t\nu}\exp\Biggl[\rho\int_{0}^{t}\left(e^{(e^{\theta}-1)\int_{u}^{t}g(s-u)ds}-1\right)du\Biggr]\\
&=e^{(e^{\theta}-1)t\nu}\exp\Biggl[\rho\int_{0}^{t}\left(e^{(e^{\theta}-1)\int_{0}^{t-u}g(s)ds}-1\right)du\Biggr]\\
&=\exp{\Biggl[(e^{\theta}-1)t\nu+\rho\int_{0}^{t}\Bigl(e^{(e^{\theta}-1)\int_{0}^{u}g(s)ds}-1\Bigr)du\Biggr]}.
\end{split}
\end{equation}
\end{proof}
\begin{remark}
From~\eqref{mgf1}, we have \beqn
\lim_{t\rightarrow\infty}\frac{1}{t}\log\ee\left[e^{\theta
N_{t}}\right]=(e^{\theta}-1)\nu+\rho\Bigl(e^{(e^{\theta}-1)\Vert
g\Vert_{L^{1}}}-1\Bigr). \feqn Then, by {G{\"a}rtner-Ellis theorem
(see for details~\citep{bordenave})}, we can say that
$(\frac{N_t}{t}\in\cdot)$ satisfies the large deviation principle
with the good rate function \beqn\label{ratef}
I(x)=\sup_{\theta\in\rr}\Biggl\{\theta
x-\Bigl[(e^{\theta}-1)\nu+\rho\Bigl(e^{(e^{\theta}-1)\Vert
g\Vert_{L^{1}}}-1\Bigr)\Bigr]\Biggr\} \feqn
\end{remark}
\begin{proof}[Proof of Lemma~\ref{theo2}]
It suffice to show that \beq
\exp\Bigl[(e^{\theta}-1)\nu+\rho\Bigl(e^{(e^{\theta}-1)\Vert
g\Vert_{L^{1}}}-1\Bigr)\Bigr]=\ee\left[e^{\theta Y}\right] \feq
where Y is an infinitely divisible random variable. To this end, Y
can be interpreted as $ Y=X+Z $. Here $X$ is a Poisson random
variable with parameter $ \nu $, $ Z=\sum_{i=1}^{R}Z_i $ is a
compound Poisson random variable, $R$ is a Poisson random variable
with parameter $ \rho $, $ Z_i $ are i.i.d. Poisson random variables
with parameter $ \Vert g\Vert_{L^{1}} $, also, $ X $ and $ Z $ are
independent thus
\begin{equation}
\begin{split}
\mathbb{E}\left[e^{\theta Y}\right]&=\mathbb{E}\left[e^{\theta X}\right]\mathbb{E}\left[e^{\theta Z}\right]\\
&=\exp\left[(e^{\theta}-1)\nu\right]\exp\left[\rho\Bigl(e^{(e^{\theta}-1)\Vert g\Vert_{L^{1}}}-1\Bigr)\right]\\
&=\exp\Bigl[(e^{\theta}-1)\nu+\rho\Bigl(e^{(e^{\theta}-1)\Vert g\Vert_{L^{1}}}-1\Bigr)\Bigr]
\end{split}
\end{equation}
Therefore, $Y$ has an infinitely divisible distribution.
\end{proof}
\begin{proof}[Proof of Lemma~\ref{theo4}]
We first consider
\beq
    e^{-t\eta(\theta)}\ee[e^{\theta N_{t}}]=\frac{\ee\left[e^{\theta N_{t}}\right]}{\exp\Bigl[(e^{\theta}-1)\nu t+\rho t\Bigl(e^{(e^{\theta}-1)\Vert g\Vert_{L^{1}}}-1\Bigr)\Bigr]}.
\feq
Using~\eqref{mgf1}, we have
\begin{equation}
\begin{split}
&\frac{\ee\left[e^{\theta N_{t}}\right]}{\exp\Bigl[(e^{\theta}-1)\nu t+\rho t\Bigl(e^{(e^{\theta}-1)\Vert g\Vert_{L^{1}}}-1\Bigr)\Bigr]}\\
&=\exp\Biggl[\rho\int_{0}^{t}\Bigl(e^{(e^{\theta}-1)\int_{0}^{u}g(s)ds}-1\Bigr)du-\rho t\Bigl(e^{(e^{\theta}-1)\Vert g\Vert_{L^{1}}}-1\Bigr)\Biggr]\\
&=\exp\Biggl[\rho\int_{0}^{t}\Bigl[e^{(e^{\theta}-1)\int_{0}^{u}g(s)ds}-e^{(e^{\theta}-1)\int_{0}^{\infty}g(s)ds}\Bigr]du\Biggr]\\
&\rightarrow\psi(\theta):=e^{\rho\varphi(\theta)}.
\end{split}
\end{equation}
To prove this, it suffices to show that
\beqn
\int_{0}^{t}\Bigl[e^{(e^{\theta}-1)\int_{0}^{u}g(s)ds}-e^{(e^{\theta}-1)\int_{0}^{\infty}g(s)ds}\Bigr]du\nonumber\\
\rightarrow
\int_{0}^{\infty}\Bigl[e^{(e^{\theta}-1)\int_{0}^{u}g(s)ds}-e^{(e^{\theta}-1)\int_{0}^{\infty}g(s)ds}\Bigr]du
\feqn as $t\rightarrow\infty,$ locally uniformly in
$\theta\in\mathbb{C}.$ It is equivalent to show that for any compact
set $K\subset\mathbb{C}$, $ \forall \epsilon>0, \exists N$ s.t. when
$ t>N $,
\begin{equation}\label{converge}
\left|\int_{t}^{\infty}\Bigl[e^{(e^{\theta}-1)\int_{0}^{u}g(s)ds}-e^{(e^{\theta}-1)\int_{0}^{\infty}g(s)ds}\Bigr]du\right|<\epsilon,
\end{equation}
for $\theta\in K,$ and we have
\begin{equation}\label{eqn10}
\begin{split}
&\left|\int_{t}^{\infty}\Bigl[e^{(e^{\theta}-1)\int_{0}^{u}g(s)ds}-e^{(e^{\theta}-1)\int_{0}^{\infty}g(s)ds}\Bigr]du\right|\\
&=\left|\int_{t}^{\infty}\Bigl[(e^{\theta}-1)e^{(e^{\theta}-1)\xi}\int_{u}^{\infty}g(s)ds\Bigr]du\right|
\end{split}
\end{equation}
by the mean value theorem, $\xi(u)\in\left(\int_{0}^{u}g(s)ds,\int_{0}^{\infty}g(s)ds
\right) $.
\begin{equation}\label{eqn11}
\begin{split}
&\left|\int_{t}^{\infty}\Bigl[(e^{\theta}-1)e^{(e^{\theta}-1)\int_{0}^{\xi}g(s)ds}\int_{u}^{\infty}g(s)ds\Bigr]du\right|\\
&<\hat{C}_K\left|\int_{t}^{\infty}\int_{u}^{\infty}g(s)dsdu\right|
\end{split}
\end{equation}
where $ \hat{C}_K=\sup_{\theta\in K}\Bigl\{(e^{\theta}-1)e^{(e^{\theta}-1)\int_{0}^{\infty}g(s)ds}\Bigr\} $. According to our assumption, $ g(t)=\mathcal O\left(\frac{1}{t^{\nu+2}}\right) $, so $ \left|\int_{t}^{\infty}\int_{u}^{\infty}g(s)dsdu\right|= \mathcal O\left(\frac{1}{t^{\nu}}\right)$ as $ t\to\infty $.

So we have  for $ \epsilon'=\epsilon/\hat{C}_K $ there exists $ N $ such that if $ t>N$,
\begin{equation}
\left|\int_{t}^{\infty}\int_{u}^{\infty}g(s)dsdu\right|<\epsilon/\hat{C}_K
\end{equation}
then for $ t>N $,
\begin{equation}
\left|\int_{t}^{\infty}\Bigl[(e^{\theta}-1)e^{(e^{\theta}-1)\int_{0}^{\xi}g(s)ds}\int_{u}^{\infty}g(s)ds\Bigr]du\right|<\epsilon
\end{equation}
so we have proved \eqref{converge}.

Thus, we conclude that
\beqn
\exp\Biggl[\rho\int_{0}^{t}\Bigl[e^{(e^{\theta}-1)\int_{0}^{u}g(s)ds}-e^{(e^{\theta}-1)\int_{0}^{\infty}g(s)ds}\Bigr]du\Biggr]
\rightarrow\psi(\theta):=e^{\rho\varphi(\theta)},
\feqn
as $t\rightarrow\infty,$ locally uniformly in $\theta\in\mathbb{C},$
where
\beq
\varphi(\theta)=\int_{0}^{\infty}\Bigl[e^{(e^{\theta}-1)\int_{0}^{u}g(s)ds}-e^{(e^{\theta}-1)\int_{0}^{\infty}g(s)ds}\Bigr]du.
\feq
Hence, $\varphi(\theta)$ is well-defined and is analytic in $\theta.$
\end{proof}

\begin{proof}[Proof of lemma \ref{theo10}]
    By definition \ref{def1}, it suffices to show that
    \beqn
    \exp\Biggl[\rho\int_{0}^{\infty}G(u)du\Biggr]-\exp\Biggl[\rho\int_{0}^{t}G(u)du\Biggr]< C_K\left(\frac{1}{t^{\nu}}\right)
    \feqn
    as $ t\to\infty$. Here $G(u)=e^{(e^{\theta}-1)\int_{0}^{u}g(s)ds}-e^{(e^{\theta}-1)\int_{0}^{\infty}g(s)ds} $. According to (\ref{eqn10}) and (\ref{eqn11}),
    \begin{equation}
    \begin{split}
    &\exp\Biggl[\rho\int_{0}^{\infty}G(u)du\Biggr]-\exp\Biggl[\rho\int_{0}^{t}G(u)du\Biggr]\\
    &=\rho^2\exp\Biggl[\rho\int_{0}^{\xi'}G(u)du\Biggr]\left(\int_{0}^{\infty}G(u)du-\int_{0}^{t}G(u)du\right)\\
    &<\bar{C}_K\left|\int_{t}^{\infty}\int_{u}^{\infty}g(s)dsdu\right|=C_K\left(\frac{1}{t^{\nu}}\right)
    \end{split}
    \end{equation}
    as $ t\to\infty. $ Here $ \xi'\in(t,\infty) $.
     Thus we have proved Lemma \ref{theo10}.

\end{proof}
\begin{proof}[Completion of the Proof of Theorem~\ref{theo1}]
From Lemma~\ref{theo2} and Lemma~\ref{theo4}, we have established
the mod-$\phi$ convergence. Hence, by Theorem 3.4.~\citep{Feray}, for
any $x>0,$ and $tx\in\nn,$ \beqn
\pp(N_{t}=tx)=e^{-tI(x)}\sqrt{\frac{I''(x)}{2\pi
t}}\Biggl(\psi(\theta^{*})+\frac{a_{1}}{t}+\frac{a_{2}}{t^{2}}+\cdots
+\frac{a_{\nu-1}}{t^{\nu-1}}+\mathcal{O}\left(\frac{1}{t^{\nu}}\right)\Biggr),
\feqn and for any $x>\nu+\rho\Vert g\Vert_{L_1},$ \beqn
\pp(N_{t}\geq tx)=e^{-tI(x)}\sqrt{\frac{I''(x)}{2\pi
t}}\frac{1}{1-e^{-\theta^{*}}}\Biggl(\psi(\theta^{*})+\frac{b_{1}}{t}+\frac{b_{2}}{t^2}+\cdots
+\frac{b_{\nu-1}}{t^{\nu-1}}+\mathcal{O}\left(\frac{1}{t^{\nu}}\right)\Biggr),
\feqn where $I(x)$ is defined in~\eqref{ratef}. This completes the
proof of Theorem~\ref{theo1}.
\end{proof}

\subsection{Proofs of the results in Section 2.2}

Since we have established the mod-$ \phi $ convergence and $ N_t $
is a lattice distribution, the proof of Theorem \ref{theo5} and
Theorem \ref{theo6} follows from Theorem 3.9 and Corollary 3.13 in
\citep{Feray}, respectively.

\noindent
{\bf Acknowledgements}\\
Youngsoo Seol is grateful to the support from the Dong-A University research grant.\\

\bibliographystyle{amsplain}

\begin{thebibliography}{100}
\bibitem[Albrecher and Asmussen(2006)]{Albrecher}
Albrecher,~H. and Asmussen,~S. 2006. \emph{Ruin probabilities and
aggregate claims distributions for shot noise Cox processes}, Scand.
Actuar. J. \textbf{2}:86--110.

\bibitem[Bordenave and Torrisi(2007)]{bordenave}
Bordenave,~C. and Torrisi,~G.L. 2007. \emph{Large deviations of
Poisson cluster processes}, Stoch.Models \textbf{24}:593-625.

\bibitem[Br\'{e}maud(2000)]{bremaud}
Br\'{e}maud,~P. 2000.\emph{An insensitivity property of Lundberg's
estimate for delayed claims}, J. Appl. Prob  \textbf{37}:914--917.

\bibitem[Cline and Hsing(1991)]{cline}
Cline, D.B.H. and Hsing, T. 1991.\emph{Large deviation probabilities
for sums and maxima of random variables with heavy or subexponential
tails}, Preprint, Texas A$\&$M University.

\bibitem[Cox(1955)]{Cox}
Cox, D. R. 1955. \emph{Some statistical methods connected with
series of events}, J. R. Stat. Soc. Ser. B Stat.Methodol
\textbf{17}(2):129--164.

\bibitem[Cox(1980)]{Cox2}
Cox, D. R. and Isham, V. 1980. \emph{Point processes}, Chapman and
Hall, London.

\bibitem[Dassios and Jang(2003)]{Dassios2}
Dassios, A. and Jang, J. 2003. \emph{Pricing of catastrophe
reinsurance and derivatives using the Cox process with shot noise
intensity}, Finance Stoch. \textbf{7}:73--95.



\bibitem[Feray, Meliot, and Nikeghbali(2013)]{Feray}
Feray, V., Meliot, P.-L, and A. Nikeghbali. 2013. \emph{Mod$-\phi$
convergence I: Normality zones and precise deviations},
arXiv:1304.2934.

\bibitem[Ganesh, Macci, and Torrisi(2007)]{Ganesh}
Ganesh, A, Macci, C, and Torrisi, G.L. 2007. \emph{A class of risk
processes with reserve-dependent premium rate: sample path large
deviations and importance sampling}, Queuing System. \textbf{55}:
83--94.

\bibitem[Gao and Yan(2008)]{Gao}
Gao, F. Q. and Yan, J. 2008. \emph{Sample path large deviations and
moderate deviations for risk processes with reinsurance}, J. Appl.
Prob  \textbf{45}: 800--817.

\bibitem[Gao and Yan(2012)]{Gao2}
Gao, F. Q. and Yan, J. 2012. \emph{Exponential martingale and large
deviations for a Cox risk process with Poisson shot noise
intensity}, J. Math. Anal. Appl.  \textbf{394}:74--83.

\bibitem[Gao and Zhu(2017)]{Gao3}
Gao, F. and Zhu, L. 2017. \emph{Precise deviations for Hawkes
processes}, arXiv:1702.02962.

\bibitem[Hellmund, Proke\v{s}ov\'{a}, and Jensen(2008)]{Hellmund}
Hellmund, G., Proke\v{s}ov\'{a}, M. and Jensen, E. B. V. 2008.
\emph{L\'{e}vy-based Cox Point processes}, Adv. in Appl. Probab.
\textbf{40}:603--629.

\bibitem[Illian et al.(2008)]{Illian}
Illian, J., Penttinen, A., Stoyan, H. and Stoyan, D. 2008.
\emph{Statistical Analysis and Modeling of Spatial Point Patterns},
Chichester:Wiley.

\bibitem[Jang and Krvavych(2004)]{Jang}
Jang, J. and Krvavych, Y. 2004. \emph{Arbitrage-free premium
calculation for extreme losses using the shot noise process and the
Esscher transform}, Insurance: math. Econom. \textbf{35}:97--111.

\bibitem[Kl\"{u}ppelberg and Mikosch(1997)]{kluppelberg}
Kl\"{u}ppelberg, C. and Mikosch, T. 1997. \emph{Large deviations of
heavy tailed random summs with applications in insurance and
finance}, J. Appl. Prob. \textbf{34}(2):293-308.

\bibitem[Macci and Torrisi(2011)]{Macci}
Macci, C. and Torrisi, G.L. 2011. \emph{ Risk processes with shot
noise Cox claim number process and reserve dependent premium rate},
Insurance: math. Econom. \textbf{48}:134--145.

\bibitem[Macci and Stabile(2006)]{Macci2}
Macci, C. and Stabile, G. 2006. \emph{Large deviation for risk
processes with reinsurance}, J. Appl. Prob. \textbf{43}:713--728.

\bibitem[McNeil et al.(2005)]{mcneil}
McNeil, A.J., Frey, R. and Embrechts, P. 2005. \emph{Quantitative
Risk Mangagment. Concepts, Techniques and Tools}, Princeton,
NJ:Princeton University Press.

\bibitem[Mikosch and Nagaev(1998)]{mikosch}
Mikosch, T. and Nagaev, A.V. 1998. \emph{Large deviations of
heavy-tailed sums with applications in insurance}, Extremes,
\textbf{1}(1):81-110.

\bibitem[M$\o$ller(2003)]{Moller}
M$\o$ller, J. 2003. \emph{Shot noise Cox processes}, Adv. in Appl.
Probab. \textbf{35}(3):614--640.

\bibitem[Ng et al.(2004)]{Ng}
Ng, K.W., Tang, Q.H., Yan, J.A. and Yang,H.L. 2004. \emph{Precise
large deviations for sums of random variables with consistently
varying tails}, J. Appl. Prob. \textbf{41}(1):93-107.

\bibitem[Schmidt(2014)]{Schmidt}
Schmidt, T. 2014. \emph{Catastrophe insurance modeled by shot noise
processes}, Risks, \textbf{2}:3-24.

\bibitem[Shen, Lin, and Zhang(2009)]{Shen}
Shen, M.X., Lin, Z.Y. and Zhang, Y. 2009. \emph{Precise large
deviations for the actual aggregate loss process}, Stoch. Anal.
Appl. \textbf{27}:1000--1013.

\bibitem[Thomas(1949)]{Thomas}
Thomas, M. 1949. \emph{A generalization of Poisson's binomial limit
for use in ecology}, Biometrika, \textbf{36}:18-25.

\bibitem[Torrisi(2004)]{Torrisi}
Torrisi, G. L. 2004. \emph{Simulating the ruin probability of risk
processes with delay in claim settlement}, Stochastic Process. Appl.
\textbf{112}:225-244.

\bibitem[Wang and Wang(2007)]{wang}
Wang, S.J. and Wang, W.S. 2007. \emph{Precise large deviations for
random variables with consistently varying tails in multi risk
models}, J. Appl. Prob. \textbf{44} (4):889-900.

\bibitem[Zhu(2013)]{Zhu5}
Zhu, L. 2013. \emph{Ruin probabilities for risk processes with
non-stationary arrivals and sub-exponential claims}, Insurance Math.
Econom. \textbf{53}:544--550.

\end{thebibliography}

\end{document}